\documentclass[a4paper,11pt]{article}
\usepackage[utf8]{inputenc}
\usepackage[T1]{fontenc}

\usepackage[affil-it]{authblk}

\usepackage{indentfirst}
\usepackage{enumitem}
\usepackage{amsmath}
\usepackage{amssymb,amsmath,amsthm}
\usepackage{ebproof}
\usepackage{aliascnt}
\usepackage{hyperref}
\usepackage[capitalise]{cleveref}
\usepackage{ragged2e}
\usepackage{adjustbox}
\usepackage [autostyle, english = american]{csquotes}
\MakeOuterQuote{"}

\usepackage[backend=biber,style=alphabetic]{biblatex}
\usepackage[affil-it]{authblk}

\usepackage{tikz-cd}
\usetikzlibrary{decorations.markings,intersections}
\tikzset{shorten <>/.style={shorten >=#1,shorten <=#1}}
\tikzset{every picture/.prefix code=\DisableQuotes}

\tikzset{%
scalearrow/.style n args={3}{
  decoration={
    markings,
    mark=at position (1-#1)/2*\pgfdecoratedpathlength
      with {\coordinate (#2);},
    mark=at position (1+#1)/2*\pgfdecoratedpathlength
      with {\coordinate (#3);},
    },
  postaction=decorate,
  }
}

\usepackage{quiver}
\usepackage[mathcal]{euscript}
\usepackage{ stmaryrd }
\usepackage{ dsfont }
\usepackage{upgreek}
\usepackage{ mathrsfs }
\usepackage{comment}

\renewenvironment{abstract}
 {\small
  \begin{center}
  \bfseries \abstractname\vspace{-.5em}\vspace{0pt}
  \end{center}
  \list{}{%
    \setlength{\leftmargin}{2mm}
    \setlength{\rightmargin}{\leftmargin}%
  }%
  \item\relax}
 {\endlist}
 
 \let\oldtheorem\newtheorem
\RenewDocumentCommand{\newtheorem}{s m o m O{}}{%
\IfBooleanTF{#1}%
{\oldtheorem{#2}{#4}}%
{\IfNoValueTF{#3}{\oldtheorem{#2}{#4}[#5]}%
{\newaliascnt{#2}{#3}%
\oldtheorem{#2}[#2]{#4}%
\aliascntresetthe{#2}}}}

\newtheorem{theorem}{Theorem}[section]
\newtheorem{proposition}[theorem]{Proposition}
\newtheorem{lemma}[theorem]{Lemma}

\theoremstyle{definition}
\newtheorem{definition}{Definition}[section]

\theoremstyle{remark}
\newtheorem{remark}{Remark}[section]
\newtheorem{example}[remark]{Example}

\title{Generalized inverse diagrams in tribes}
\author{El Mehdi Cherradi}
\affil{IRIF - CNRS - Universit\'e Paris Cit\'e \\MINES ParisTech - Universit\'e PSL}
\date{}

\begin{document}

\maketitle

\begin{abstract}
 Starting from a generalized direct category $R$, we construct an absolutely dense functor $\mathbf{D}_r \to R$ with domain a strict direct category. Given any tribe $\mathcal{T}$, we leverage this construction to provide a tribe structure on a subcategory of fibrant diagrams in $\mathcal{T}^{R^{op}}$, assuming some finiteness condition on $R$.
\end{abstract}

\tableofcontents

\newpage 

\section*{Introduction}
\addcontentsline{toc}{section}{Introduction}

Given a combinatorial model category $\mathcal{M}$ and any small category $C$, both the projective and injective model structure on the category of diagrams $\mathcal{M}^C$ exist. If $C$ is a Reedy category, a third model structure, the Reedy model structure, which is in general different from the other two, always exists even when dropping the combinatoriality assumption on $\mathcal{M}$.

For fibration categories, the situation is a bit different. Indeed, given a fibration category $\mathcal{F}$, the "natural" fibration category structure on a category of $C$-shaped diagrams in $\mathcal{F}$ is arguably the Reedy model structure, assuming that $C$ is an inverse category. By this, we mean that the notion of Reedy model structure adapts in a straightforward manner to yield a definition of Reedy fibrant diagram, forming a category where it is possible to factor any map suitably as required from the axioms of a fibration category. If $C$ is not assumed to be an inverse category, it is still possible to endow the category $\mathcal{F}^C$ with the structure of a fibration category, where both the fibrations and the weak equivalences are the pointwise ones: this is established in \cite[Theorem 9.5.5]{radulescu2006cofibrations}, relying on the Reedy fibration category structure on $\mathcal{F}^{\Delta' \downarrow C}$, where $\Delta'$ is the direct category spanned by the injective maps in the simplex category $\Delta$.
In \cite[Théorème 6.17]{cisinski2010categories}, another fibration category structure is established on a subcategory of "$\tau$-fibrant" diagrams in $\mathcal{F}^C$, working in a slightly different manner with the usual functor $\tau : \Delta' \downarrow C \to C$.

Finally, in the case of tribes, it seems that a tribe structure on a category of diagram $\mathcal{T}^C$, where $\mathcal{T}$ is a tribe and $C$ a small category, has only been considered in the literature when $C$ is an inverse category, in which case the notion of Reedy fibrancy can be used (see \cite[Lemma 2.22]{ks2019internal}).

In the present document, we will be interested in the notion of generalized inverse/direct category, which aims at relaxing the definition of an inverse/direct category in order to allow non-identity isomorphisms. These notions are essentially special cases of generalized Reedy categories, which have been introduced by Cisinski in \cite{cisinski2006prefaisceaux}, then slightly generalized and further discussed by Berger and Moerdijk in \cite{berger2011extension}.

Our objective is to provide a tribe structure on a subcategory of fibrant diagrams $\mathcal{T}^C$, assuming $C$ to be a generalized inverse category.

\section{The "unrolling" construction}

In this section, we consider a generalized direct category $R$, in the sense of the following definition:

\begin{definition}
 We say that a category $R$ is a generalized direct category when there exists a degree function $\mathbf{deg} : \mathbf{Ob}(R) \to \alpha$ for some ordinal $\alpha$ such that invertible morphisms preserve the degree, and non-invertible ones strictly raise it.
 
 We say that $R$ is a generalized inverse category when $R^{op}$ is generalized direct.
\end{definition}

We also consider the following additional data:
\begin{itemize}
 \item A functor $r : R_0 \to R$ from a (strict) direct category, such that every arrow $f : a \to b$ in $R$ lifts to an arrow $k : x \to y$ in $R_0$ up to isomorphism, i.e., such that there is an isomorphism $w : a \simeq r(x)$ and $w' : b \simeq r(y)$ fitting in a commutative square:

\[\begin{tikzcd}[ampersand replacement=\&]
	a \&\& b \\
	\\
	{r(x)} \&\& {r(y)}
	\arrow["f", from=1-1, to=1-3]
	\arrow["w"', from=1-1, to=3-1]
	\arrow["{w'}", from=1-3, to=3-3]
	\arrow["{c(k)}"', from=3-1, to=3-3]
\end{tikzcd}\]
(in other words, the functor $r^\rightarrow : R_0^\rightarrow \to R^\rightarrow$ is required to be essentially surjective on objects).
\end{itemize}

Note that there is a canonical choice, namely considering the subcategory $R_0 \subset R$ spanned by the non-invertible morphisms that are not identities. Since $R$ is generalized direct, this class of arrows is stable under composition, so we indeed get a subcategory.

\begin{example}
 Our main source of examples is that of categories of semi-cubes. In \cite{campion2023cubical}, several variations on the category of cubes are considered. For a given category of cubes $\square$ without isomorphisms, there is usually a counterpart  \textit{with symmetries} $\square_s$. Since we are only interested in direct categories, we consider the subcategories that do not contain the degeneracies.
 
 Hence, concrete examples include the category of semi-cubes $\square_\sharp$, which is the free monoidal category generated by two projections. The corresponding category of semi-cubes with symmetries $\square^s_\sharp$ is the free symmetric monoidal category generated by $\square_\sharp$.
\end{example}
 
 \begin{definition}
  Write $\mathbf{F}$ for the free category comonad on $\mathbf{Cat}$ (on reflexive graphs, or even on graphs if $r$ is surjective on objects), then form the pushout as in the diagram below:

\[\begin{tikzcd}[ampersand replacement=\&]
	{\mathbf{F}(R_0)} \&\& {\mathbf{F}(R)} \\
	\\
	{R_0} \&\& {\mathbf{F}_\simeq(R)} \\
	\&\&\& R
	\arrow[from=1-1, to=1-3]
	\arrow[from=1-1, to=3-1]
	\arrow[from=1-3, to=3-3]
	\arrow[curve={height=-6pt}, from=1-3, to=4-4]
	\arrow[from=3-1, to=3-3]
	\arrow[curve={height=6pt}, from=3-1, to=4-4]
	\arrow["\ulcorner"{anchor=center, pos=0.125, rotate=180}, draw=none, from=3-3, to=1-1]
	\arrow["{p_0}", dashed, from=3-3, to=4-4]
\end{tikzcd}\]
 Finally, write $\mathbf{D}_{R}$ for the full subcategory of the twisted arrow category $\mathbf{Tw}(\mathbf{F}_\simeq(R))$ spanned by the objects consisting of an arrow $x \to y$ of the form $x \to z \to y$ where $x \to z$ comes from a map in $R_0$, and $z \to y$ corresponds to a "free" isomorphism in $R$ or is an identity arrow (that is, an arrow obtained by taking the image under $\mathbf{F}(R) \to \mathbf{F}_\simeq(R)$ of a path of length one in $R$ whose underlying arrow is an isomorphism). It comes with a projection $p : \mathbf{D}_{R} \to R$.
 \end{definition}
 
 \begin{lemma}
  $\mathbf{D}_{R}$ is a (strict) direct category.
 \end{lemma}
 
 \begin{proof}
 Consider an object of $\mathbf{D}_{R}$ given by a map $f : x \to z \to y$. 
  We take the degree of such an object to be $\mathbf{deg}(y) - \mathbf{deg}(x) + k$, where $k = 0$ if $z \to y$ is an identity arrow, and $k = 1$ otherwise.
  
  With this definition, it is clear that non-identity arrows must stricly increase the degree (either because they stricly increase $\mathbf{deg}(y)$, they stricly decrease $\mathbf{deg}(x)$, or they add a "free" isomorphism, strictly increasing $k$).
\end{proof}
 
 \begin{lemma}
  The canonical functor $p : \mathbf{D}_r \to R$ is absolutely dense (i.e., the precomposition functor $p^* : \mathbf{Set}^{R} \to \mathbf{Set}^{\mathbf{D}_r}$ is fully faithful).
 \end{lemma}
 
 \begin{proof}
  We rely on the following criterion, established in \cite[Theorem 1.1]{adamek2001functors}: a functor $F : A \to B$ is absolutely dense if and only if, for every morphism $f : b \to b'$ in $B$, the category of $F$-factorization,  that has objects the tuples $(a, b \to F(a), F(a) \to b)$ yielding a factorization of $f$ and has morphisms the maps $a \to a'$ making the corresponding diagrams commute, is a (non-empty) connected category.
  
  Given a morphism $f :  x \to y$ in $R$, there is an obvious factorization through the object of $\mathbf{D}_{R}$ given by the identity arrow $y \to y$. Thus, the category of factorization is not empty. Next, consider a factorization $f : x \to p(Z) \to y$, where $Z : z \to z'$ is a arrow factoring as $[w] \circ Z_0$, and $[w]$ is a "free" isomorphism. Then, there is a morphism $Z_0 \to Z$ in the twisted arrow category, and a diagram
\[\begin{tikzcd}[ampersand replacement=\&]
	\&\& {p(Z)} \&\& \\
	x \&\&\&\& y \\
	\&\& {p(Z_0)}
	\arrow["\delta", from=1-3, to=2-5]
	\arrow["{f'}", from=2-1, to=1-3]
	\arrow["{p_0(w)^{-1} \circ  f'}"', from=2-1, to=3-3]
	\arrow["{p_0(w)}", dashed, from=3-3, to=1-3]
	\arrow["{\delta  \circ p_0(w)}"', from=3-3, to=2-5]
\end{tikzcd}\]
where $w$ is the isomorphism in $R$ corresponding to $[w]$. We hence have a morphism in the category of factorization of $f$. In the case where $Z$ does not end with a "free" isomorphism, then we can just take $Z_0 = Z$, and $w = id_{z'}$.

  Next, by assumption on $R_0$, the morphism $\delta : z' \to y$ can be lifted to $R_0$ (and, then, to $\mathbf{F}_\simeq(R)$), as an arrow $\delta_s$ that comes with a commutative square:

\[\begin{tikzcd}[ampersand replacement=\&]
	{z'} \&\& y \\
	\\
	{z''} \&\& {y'}
	\arrow["\delta", from=1-1, to=1-3]
	\arrow["{w'}"', from=1-1, to=3-1]
	\arrow["v", from=1-3, to=3-3]
	\arrow["{r(\delta_s)}"', from=3-1, to=3-3]
\end{tikzcd}\]

We define the object $Z_1 : z \to z''$ whose underlying arrow is the composite $[w' \circ w] \circ Z_0$, fitting in a diagram, which is a zig-zag in the twisted arrow category $\mathbf{D}_{R}$,
  
\[\begin{tikzcd}[ampersand replacement=\&]
	z \&\& z \&\& {z''} \&\& {z''} \&\& y \\
	\\
	{z'} \&\& {z''} \&\& {z''} \&\& y \&\& y
	\arrow["{Z_0}"', from=1-1, to=3-1]
	\arrow["{id_z}"', from=1-3, to=1-1]
	\arrow["{Z_1}", from=1-3, to=1-5]
	\arrow["{Z_1}"', from=1-3, to=3-3]
	\arrow["{id_{z''}}", from=1-5, to=3-5]
	\arrow["{id_{z''}}"', from=1-7, to=1-5]
	\arrow["{[v^{-1}] \circ c(\delta_s)}", from=1-7, to=1-9]
	\arrow["{Z_2}"', from=1-7, to=3-7]
	\arrow["{id_y}", from=1-9, to=3-9]
	\arrow["{[w' \circ w]}"', from=3-1, to=3-3]
	\arrow["{id_{z''}}", from=3-5, to=3-3]
	\arrow["{[v^{-1} ]\circ c(\delta_s)}"', from=3-5, to=3-7]
	\arrow["{id_y}", from=3-9, to=3-7]
\end{tikzcd}\]
and where we also defined $Z_2$ as another intermediate step.

This finally yields a zig-zag of factorization,
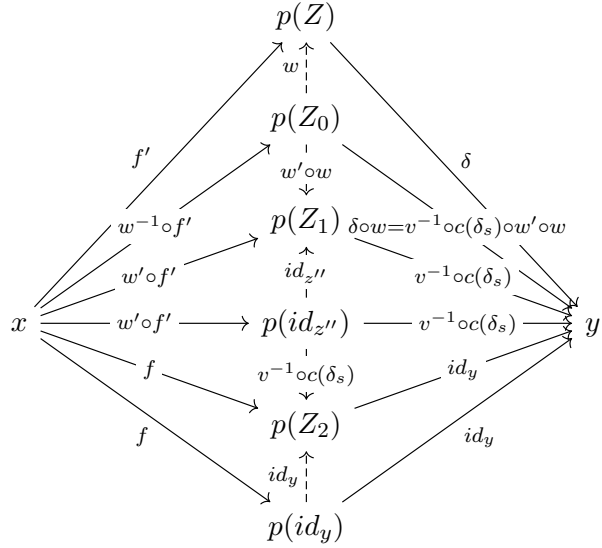
\[\begin{tikzcd}[ampersand replacement=\&]
	\&\&\& {p(Z)} \\
	\&\&\& {p(Z_0)} \\
	\&\&\& {p(Z_1)} \\
	x \&\&\& {p(id_{z''})} \&\&\& y \\
	\&\&\& {p(Z_2)} \\
	\&\&\& {p(id_y)}
	\arrow["\delta", from=1-4, to=4-7]
	\arrow["w", dashed, from=2-4, to=1-4]
	\arrow["{w' \circ w}"{description}, dashed, from=2-4, to=3-4]
	\arrow["{\delta \circ w = v^{-1} \circ r(\delta_s) \circ w' \circ w}"{description}, from=2-4, to=4-7]
	\arrow["{v^{-1} \circ r(\delta_s)}"{description}, from=3-4, to=4-7]
	\arrow["{f'}", from=4-1, to=1-4]
	\arrow["{w^{-1} \circ f'}"{description}, from=4-1, to=2-4]
	\arrow["{w' \circ f'}"{description}, from=4-1, to=3-4]
	\arrow["{w' \circ f'}"{description}, from=4-1, to=4-4]
	\arrow["f"{description}, from=4-1, to=5-4]
	\arrow["f"', from=4-1, to=6-4]
	\arrow["{id_{z''}}"{description}, from=4-4, to=3-4]
	\arrow["{v^{-1} \circ r(\delta_s)}"{description}, from=4-4, to=4-7]
	\arrow["{v^{-1} \circ r(\delta_s)}"{description}, from=4-4, to=5-4]
	\arrow["{id_y}"{description}, from=5-4, to=4-7]
	\arrow["{id_y}"', from=6-4, to=4-7]
	\arrow["{id_y}", dashed, from=6-4, to=5-4]
\end{tikzcd}\]
thus proving that the category of factorization of $f$ is connected.
 \end{proof}

\section{Diagrams in a tribe}

In this section, we assume that $R$ has moreover finitely many objects in each degree, and finitely many isomorphisms in each degree as well.

 From now on, we also fix a tribe $\mathcal{T}$. Our objective is to provide a tribe structure for some category of "fibrant" diagrams in $\mathcal{T}^{R^{op}}$. The notion of fibrations we introduce below is reminiscent of the definition of $\tau$-fibrations considered in \cite[Théorème 6.17]{cisinski2010categories}.

 \begin{definition}
  We define the class of $p$-fibrations as the class of morphisms $m : F \to F'$ in $\mathcal{T}^{R^{op}}$ such that $p^* m$ is a Reedy fibration in $\mathcal{T}^{\mathbf{D}_r^{op}}$.
 \end{definition}
 
 \begin{definition}
  We define $\mathcal{T}^{R^{op}}_f$ as the full subcategory of $\mathcal{T}^{R^{op}}$ spanned by the $p$-fibrant diagrams, and we define the fibrations between two such diagrams to be the $p$-fibrations.
 \end{definition}
 
 We will need the following definition and theorem to reach our goal:
 
 \begin{definition}[Definition 2.12 and 2.15 in \cite{hirschhorn2019functors}]
 \label{fibering_def}
  Consider a functor $G : \mathcal{C} \to \mathcal{D}$ between Reedy categories. For every object $\alpha$ in $\mathcal{C}$, for every object $\beta$ in $\mathcal{D}$, and map $\sigma : \alpha \to G(\beta)$ in the distinguished subcategory $\mathcal{D}_+$, define the category $\mathbf{Fact}_{\mathcal{C}_+}(\alpha, \sigma)$ whose objects are the factorization of $\sigma$ given by diagrams

\[\begin{tikzcd}[ampersand replacement=\&]
	\alpha \&\& {G(\gamma)} \\
	\& {G(\beta)}
	\arrow["\mu", from=1-1, to=1-3]
	\arrow["\sigma"', from=1-1, to=2-2]
	\arrow["{G(\nu)}", from=1-3, to=2-2]
\end{tikzcd}\]
where $\nu$ is an arrow in $\mathcal{C}_+$,
and whose maps between two such factorization $(\mu,\nu)$ and $(\mu',\nu')$ are the arrows $\tau : \gamma \to \gamma'$ making the following two diagrams commute:
\[\begin{tikzcd}[ampersand replacement=\&]
	\gamma \&\& {\gamma'} \&\& \alpha \\
	\& \beta \&\& {G(\gamma)} \&\& {G(\gamma')}
	\arrow["\tau", from=1-1, to=1-3]
	\arrow[from=1-1, to=2-2]
	\arrow[from=1-3, to=2-2]
	\arrow[from=1-5, to=2-4]
	\arrow[from=1-5, to=2-6]
	\arrow["{G(\tau)}"', from=2-4, to=2-6]
\end{tikzcd}\]

  The functor $G$ is said to be cofibering when all the categories $\mathbf{Fact}_{\mathcal{C}_+}(\alpha, \sigma)$ are either empty or connected. $G$ is said to be fibering when $G^{op} : \mathcal{C}^{op} \to \mathcal{D}^{op}$ is cofibering.
 \end{definition}
 
 \begin{theorem}[Adapted from {\cite[Theorem 4.2]{hirschhorn2019functors}}]
 \label{fibering_thm}
  Let $G : \mathbf{C} \to \mathbf{D}$ be a fibering functor between inverse categories, and $\mathcal{T}$ be a tribe. Then the precomposition functor $$\mathcal{T}^\mathbf{D} \to \mathcal{T}^\mathbf{C}$$ maps Reedy fibrations to Reedy fibration. In particular, the induced functor $$\mathcal{T}^\mathbf{D}_r \to \mathcal{T}^\mathbf{C}_R$$ between the tribes of Reedy fibrant diagrams is a morphism of tribe.
 \end{theorem}
 
 \begin{proof}
  In \cite{hirschhorn2019functors}, the authors consider model categories rather than tribes, and prove that the precomposition functor is right Quillen under the assumption on $G$. However, the proof they give also applies to the settings of tribes in that it only relies on axioms that also hold for tribes. More precisely, the pullbacks they form are always pullbacks along fibrations, which are also available in tribes. The additional properties that are made use of are the stability of the class of fibrations by composition and pullback, as well as the fact that all isomorphisms are fibrations, which is also true in any tribe. Therefore, we can conclude that the precomposition functor $$\mathcal{T}^\mathbf{D} \to \mathcal{T}^\mathbf{C}$$ maps Reedy fibrations to Reedy fibration. Since the anodyne maps in $\mathcal{T}^\mathbf{C}_R$ and $\mathcal{T}^\mathbf{D}_r$ are the pointwise anodyne maps (see \cite[Lemma 2.22]{ks2019internal}), the precomposition functor also maps anodyne maps to anodyne maps. This functor also preserves limits, in particular the terminal object and the pullbacks along fibrations. This means that we indeed get a morphism of tribe: $$\mathcal{T}^\mathbf{D}_r \to \mathcal{T}^\mathbf{C}_R$$
\end{proof}
 
 The following proposition is the key step in order to establish that the notion of $p$-fibration is well-behaved.
 
 \begin{proposition}
  The right Kan extension functor $$p_* : \mathcal{T}^{\mathbf{D}_r^{op}}_f \to \mathcal{T}^{R^{op}}$$ is well-defined and takes values in fibrant diagrams. Moreover, $p_*$ maps fibrations to $p$-fibrations and pointwise anodyne maps to pointwise anodyne maps.
 \end{proposition}
 
 \begin{proof}
  First, observe that $p_*$ can be computed pointwise as limits of finite fibrant diagrams. Indeed, the Kan extension is computed pointwise from (the dual of) the following comma category square:

\[\begin{tikzcd}[ampersand replacement=\&]
	{P_\alpha} \&\& {p \downarrow p} \&\& {\mathbf{D}_r} \\
	\\
	{*} \&\& {\mathbf{D}_r} \&\& R
	\arrow[from=1-1, to=1-3]
	\arrow[from=1-1, to=3-1]
	\arrow["\ulcorner"{anchor=center, pos=0.125}, draw=none, from=1-1, to=3-3]
	\arrow["{\pi_0}", from=1-3, to=1-5]
	\arrow["{\pi_1}"', from=1-3, to=3-3]
	\arrow[between={0.3}{0.7}, Rightarrow, from=1-5, to=3-3]
	\arrow["p", from=1-5, to=3-5]
	\arrow["\alpha"', from=3-1, to=3-3]
	\arrow["p"', from=3-3, to=3-5]
\end{tikzcd}\]
where $P_\alpha$ is finite since the degree (in $\mathbf{D}_r$) of any of its objects is bounded by $n+1$, where $n := \mathbf{deg}(p(\alpha))$, and since there are finitely many objects of each degree.

Next, we claim that the projection $\pi_0 : p \downarrow p \to \mathbf{D}_r$ is a cofibering Reedy functor (\cref{fibering_def}).
To show this, consider a map $f : \alpha \to \pi_0 X$. Since $\pi_0$ is a Grothendieck fibration, there exists a cartesian lifting for $f$, providing a trivial factorization for $f$, so that the category of factorization of $f$ is non-empty. Moreover, the cartesian lifting is of the following form,
\[\begin{tikzcd}[ampersand replacement=\&]
	x \&\& \bullet \\
	\\
	y \&\& \bullet
	\arrow["\alpha"', from=1-1, to=3-1]
	\arrow["{X_t}", from=1-3, to=3-3]
	\arrow[from=3-1, to=3-3]
\end{tikzcd}\]
where $X_t$ is the second component of the object $X \in p \downarrow p$.

Now, any factorization as in the diagram below,
\[\begin{tikzcd}[ampersand replacement=\&]
	\alpha \&\& {\pi_0 Y} \\
	\& {\pi_0 X}
	\arrow["g", from=1-1, to=1-3]
	\arrow["f"', from=1-1, to=2-2]
	\arrow["{\pi_0 H}", from=1-3, to=2-2]
\end{tikzcd}\]
where $H$ is a non-identity arrow, corresponds to a diagram as follows:
\[\begin{tikzcd}[ampersand replacement=\&]
	\&\&\&\&\& \bullet \\
	\&\&\& \bullet \\
	x \&\&\&\& \bullet \& \bullet \\
	\&\& \bullet \& \bullet \\
	y \&\&\&\& \bullet \\
	\&\& \bullet
	\arrow["{Y_t}", from=1-6, to=3-6]
	\arrow[from=2-4, to=3-1]
	\arrow["{Y_s}"', from=2-4, to=4-4]
	\arrow["\alpha"', from=3-1, to=5-1]
	\arrow[from=3-5, to=1-6]
	\arrow["{X_t}", from=3-5, to=5-5]
	\arrow[from=3-6, to=5-5]
	\arrow[from=4-3, to=2-4]
	\arrow[from=4-3, to=3-1]
	\arrow["{X_s}"', from=4-3, to=6-3]
	\arrow[from=4-4, to=3-6]
	\arrow[from=4-4, to=6-3]
	\arrow[from=5-1, to=4-4]
	\arrow[from=5-1, to=6-3]
	\arrow[from=6-3, to=5-5]
\end{tikzcd}\]

It is not difficult to complete this diagram in order by a zig-zag leading to the cartesian lifting of $f$, as we do below,

\[\begin{tikzcd}[ampersand replacement=\&]
	\&\&\& \bullet \\
	\& x \&\&\& \bullet \\
	\&\& x \& \bullet \\
	\& y \&\&\& \bullet \& \bullet \\
	\&\& y \& \bullet \\
	x \&\&\&\& \bullet \& \bullet \\
	\&\& \bullet \& \bullet \\
	y \&\&\&\& \bullet \\
	\&\& \bullet
	\arrow[from=1-4, to=2-5]
	\arrow["{X_t}"', from=1-4, to=3-4]
	\arrow["id", from=2-2, to=3-3]
	\arrow["\alpha"', from=2-2, to=4-2]
	\arrow["{Y_t}", from=2-5, to=4-5]
	\arrow["\alpha", from=3-3, to=5-3]
	\arrow[from=3-3, to=6-1]
	\arrow[from=4-2, to=3-4]
	\arrow[from=4-5, to=3-4]
	\arrow["id"'{pos=0.3}, from=4-5, to=6-6]
	\arrow[from=4-6, to=2-5]
	\arrow["{Y_t}", color={rgb,255:red,41;green,41;blue,163}, from=4-6, to=6-6]
	\arrow["id", from=5-3, to=4-2]
	\arrow[from=5-3, to=4-5]
	\arrow[from=5-3, to=7-4]
	\arrow[from=5-4, to=3-3]
	\arrow[color={rgb,255:red,41;green,41;blue,163}, from=5-4, to=6-1]
	\arrow["{Y_s}"', color={rgb,255:red,41;green,41;blue,163}, from=5-4, to=7-4]
	\arrow["\alpha"', color={rgb,255:red,41;green,41;blue,163}, from=6-1, to=8-1]
	\arrow[color={rgb,255:red,41;green,41;blue,163}, from=6-5, to=4-6]
	\arrow["{X_t}", color={rgb,255:red,41;green,41;blue,163}, from=6-5, to=8-5]
	\arrow[color={rgb,255:red,41;green,41;blue,163}, from=6-6, to=8-5]
	\arrow[color={rgb,255:red,41;green,41;blue,163}, from=7-3, to=5-4]
	\arrow[color={rgb,255:red,41;green,41;blue,163}, from=7-3, to=6-1]
	\arrow["{X_s}"', color={rgb,255:red,41;green,41;blue,163}, from=7-3, to=9-3]
	\arrow[color={rgb,255:red,41;green,41;blue,163}, from=7-4, to=6-6]
	\arrow[color={rgb,255:red,41;green,41;blue,163}, from=7-4, to=9-3]
	\arrow[from=8-1, to=5-3]
	\arrow[color={rgb,255:red,41;green,41;blue,163}, from=8-1, to=7-4]
	\arrow[color={rgb,255:red,41;green,41;blue,163}, from=8-1, to=9-3]
	\arrow[color={rgb,255:red,41;green,41;blue,163}, from=9-3, to=8-5]
\end{tikzcd}\]
hence proving the category of factorization to be connected.

If follows that this induces by precomposition a morphisms of tribes $\mathcal{T}^{\mathbf{D}_r^{op}}_f \to \mathcal{T}^{(p \downarrow p)^{op}}_f$ (by \cref{fibering_thm}). We claim that the right Kan extension along the second projection $\pi_1$ takes value in fibrant diagrams, maps Reedy fibrations to Reedy fibrations and pointwise anodyne maps to pointwise anodyne maps. The first two points follow from \cite[Theorem 9.4.3 (2)]{radulescu2006cofibrations}. The third one is proved similarly to \cite[Theorem 9.3.5 (2c)]{radulescu2006cofibrations}, which applies to weak equivalences but only relies on stability under composition and pullback along fibrations of this class, which is also enjoyed by the class of anodyne maps in a tribe, as well as the so-called Gluing lemma for tribes (\cite[Lemma 2.19]{ks2019internal}).

This allows us to conclude that $p_*$ takes value in $p$-fibrant diagrams and maps fibrations to $p$-fibrations. Moreover, a map $X \to Y$ in $\mathcal{T}^{R^{op}}$ is a pointwise anodyne map as soon as its image under $p^*$ is a pointwise anodyne map since $p$ is surjective on objects, therefore $p_*$ also maps pointwise anodyne maps to pointwise anodyne maps.
 \end{proof}
 
 We can now establish our main result:
 
 \begin{theorem}
 \label{main_result}
  $\mathcal{T}^{R^{op}}_f$ enjoys the structure of a tribe.
 \end{theorem}
 
 \begin{proof}
  The class of $p$-fibrations is clearly stable under pullback and composition. Moreover, pointwise anodyne maps are stable under pullback along $p$-fibrations because the pointwise anodyne maps in $\mathcal{T}^{\mathbf{D}_{\square^s}^{op}}_R$ are the anodyne maps. 
  Given a map $F : X \to Y$ between diagrams in $\mathcal{T}^{R^{op}}_f$ , we can take its image under $p^*$, then factor it into a pointwise anodyne map followed by a fibration. Taking the image through $p_*$ yields back a factorization $X \simeq p_* p^* X \to Z \to Y \simeq p_* p^* Y$ of $F$ as a pointwise anodyne map followed by a $p$-fibration.
  We still need to check that the pointwise anodyne maps are the anodyne maps in $\mathcal{T}^{R^{op}}_f$. It is clear that pointwise anodyne maps are anodyne: this is because any lifting problem of the form

\[\begin{tikzcd}[ampersand replacement=\&]
	A \&\& X \\
	\\
	B \&\& Y
	\arrow[from=1-1, to=1-3]
	\arrow["\sim"{description}, hook', from=1-1, to=3-1]
	\arrow[two heads, from=1-3, to=3-3]
	\arrow[from=3-1, to=3-3]
\end{tikzcd}\]
where $X \to Y$ is a $p$-fibration, and $A \to B$ a pointwise anodyne maps, we may take the image of this problem through $p^*$, take a solution in $\mathcal{T}^{\mathbf{D}_r^{op}}_f$, and take again the image under $p_*$, which yields back the original problem together with the desired lift since $p_* \circ p^* \simeq id_{\mathcal{T}^{R^{op}}}$. The usual retract argument proves the converse: if $A \to B$ is anodyne, we can factor it as a pointwise anodyne $A \to C$ map followed by a $p$-fibration $C \to B$, and solve the lifting problem below,
\[\begin{tikzcd}[ampersand replacement=\&]
	A \&\& C \\
	\\
	B \&\& B
	\arrow[from=1-1, to=1-3]
	\arrow["\sim"{description}, hook', from=1-1, to=3-1]
	\arrow[two heads, from=1-3, to=3-3]
	\arrow["h"', dashed, from=3-1, to=1-3]
	\arrow[from=3-1, to=3-3]
\end{tikzcd}\]
hence exhibiting $A \to B$ as a retract of $A \to C$, so that $A \to B$ is also a pointwise anodyne map.
\end{proof}

\begin{proposition}
 If $\mathcal{T}$ is a $\pi$-tribe, then so is $\mathcal{T}^{R^{op}}_f$.
\end{proposition}

\begin{proof}
 Given two $p$-fibrations $f : A \to B$ and $g : B \to C$, we claim that $p_* \Pi_{p^* g}(p^* f)$ defines an internal product of $f$ along $g$. First observe that the resulting map $\Pi_g f \to C$ is indeed a $p$-fibration.
 
 We still need to check that the object so-defined enjoys the correct universal property. Consider a map $D \to C$, and a map $D \times_C B \to A$. Then, $p^*(D \times_C B) \simeq p^*(D) \times_{p^*C}  p^*B \to p^*A$ factors through the evaluation map $\epsilon :  \Pi_{p^* g}(p^* f) \to p^*A$ via $v \times_{p^*C} p^*B$ for a uniquely defined map $v:  p^*D \to  \Pi_{p^* g}(p^* f)$. Taking the image under $p_*$, this yields a factorization of $D \times_C B \to A$ through the evaluation $p_* \epsilon : p_* \Pi_{p^* g}(p^* f) \to A$ via $p_* v : D \to p_* \Pi_{p^* g}(p^* f)$, which is the unique map with this property since $p^*$ is fully faithful.
 
 Finally, it is clear that anodyne maps are mapped to anodyne maps by the internal product functor $\Pi_g$ since the tribe of Reedy fibrant diagrams of shape $\mathbf{D}_r^{op}$ in $\mathcal{T}$ is known to be a $\pi$-tribe when $\mathcal{T}$ is a $\pi$-tribe.
\end{proof}

\section{Examples}

In this section, we consider a group $G$ (seen as one-object category). With $G_0 := *$ the terminal category, and with the unique functor $G_0 \to G$, the condition of the first section is satisfied. 

The category $\mathbf{D}_G$, constructed from $G_0 \to G$, has objects the arrows $g : * \to *$ corresponding to a "free" automorphism $g$ of the object $*$ of $G$, that is to the elements of $G$, except for the identity element that corresponds to the identity arrow $id_* : * \to *$.
It is easy to see that there are no morphisms between two of these arrows corresponding to non-identity elements of $G$, and that for any such element $g \in G$, there are exactly two arrows from the identity, as pictured below.

\[\begin{tikzcd}[ampersand replacement=\&]
	{*} \&\& {*} \&\& {*} \&\& {*} \\
	\\
	{*} \&\& {*} \&\& {*} \&\& {*}
	\arrow["{id_*}", from=1-1, to=3-1]
	\arrow["g"', from=1-3, to=1-1]
	\arrow["g"', from=1-3, to=3-3]
	\arrow["{id_*}"', from=1-5, to=3-5]
	\arrow["{id_*}"', from=1-7, to=1-5]
	\arrow["g", from=1-7, to=3-7]
	\arrow["{id_*}"', from=3-1, to=3-3]
	\arrow["g"', from=3-5, to=3-7]
\end{tikzcd}\]

In particular, when $G$ has only one non-identity element, the category $\mathbf{D}_G^{op}$ is given by two objects and two parallel arrows between them.

Given a tribe $\mathcal{T}$, we have established in \cref{main_result} that $\mathcal{T}^{G^{op}}$ also has a tribe structure. Given an object $X$ in $\mathcal{T}^{G^{op}}$, which is nothing but an object $x$ of $\mathcal{T}$ together with an action of $G(^{op})$ on it, the object $p^* X$ of $\mathcal{T}^{\mathbf{D}_G^{op}}$ is the diagram
\[\begin{tikzcd}[ampersand replacement=\&]
	\&\& x \\
	\\
	x \&\& x \& {...} \\
	\&\& {...}
	\arrow["{g_1}"', shift right=2, from=1-3, to=3-3]
	\arrow["{id_x}", shift left=2, from=1-3, to=3-3]
	\arrow["{g_0}", shift left=2, from=3-1, to=3-3]
	\arrow["{id_x}"', shift right=2, from=3-1, to=3-3]
\end{tikzcd}\]
where $g_i$ are the elements of $G$.

The Reedy fibrancy criterion for $p^* X$ is trivially verified for the unique object of degree $0$ (corresponding to the identity arrow), since the matching object is the terminal object of $\mathcal{T}$. For the objects of the degree $1$, i.e. for all the other objects, the matching objects are binary products, and the matching map is of the form
\[\begin{tikzcd}[ampersand replacement=\&]
	x \&\& {x \times x}
	\arrow["{<id_x,g_i>}"', from=1-1, to=1-3]
\end{tikzcd}\]
and the Reedy fibrancy criterion asserts that this is a fibration.

For instance, if $G$ has only one non-identity element, if $\mathcal{T}$ is the tribe of small categories with the fibrations being the isofibrations, and if $X := (C,id_C)$ is a category equipped with the identity automorphism, the fibrancy criterion boils down to the diagonal functor $$C \to C \times C$$ being an isofibration. This is equivalent to the statement that $C$ admits no non-trivial isomorphisms ($C$ is a \textit{gaunt} category).

In particular, this shows that the underlying fibration category structure differs from the one of \cite{radulescu2006cofibrations} where the fibrations are pointwise.

\newpage

\begingroup
\setlength{\emergencystretch}{.5em}
\RaggedRight
\printbibliography
\endgroup

\end{document}